\documentclass[11pt, reqno]{amsart}

\usepackage{amsmath,amssymb,amsthm, epsfig}
\usepackage{hyperref}
\usepackage{amsmath}
\usepackage{amssymb}
\usepackage{color}
\usepackage{ulem}
\usepackage{epsfig}
\usepackage[mathscr]{eucal}
\usepackage[latin1]{inputenc}

\newtheorem{theorem}{Theorem}
\newtheorem{definition}{Definition}
\newtheorem{assumption}{A}

\newtheorem{lemma}{Lemma}
\newtheorem{proposition}{Proposition}

\newtheorem{remark}{Remark}

\date{}
\numberwithin{equation}{section}
\numberwithin{theorem}{section}
\numberwithin{lemma}{section}
\numberwithin{corollary}{section}
\numberwithin{remark}{section} 
\numberwithin{proposition}{section}
\numberwithin{definition}{section}

\newcommand{\tr}{\operatorname{Tr}}
\newcommand{\loglip}{\operatorname{Log-Lip}}
\renewcommand{\div}{\operatorname{div}}

\def \R {\mathbb{R}}

\begin{document}

\title[Existence and improved regularity for a nonlinear system]{Existence and improved regularity for a nonlinear system with collapsing ellipticity}

\author[E. A. Pimentel]{Edgard A. Pimentel}
\address{Department of Mathematics, Pontifical Catholic University of Rio de Ja\-nei\-ro -- PUC-Rio, 22451-900, G\'avea, Rio de Janeiro-RJ, Brazil}
\email{pimentel@puc-rio.br}

\author[J.M. Urbano]{Jos\'e Miguel Urbano}
\address{University of Coimbra, CMUC, Department of Mathematics, 3001-501 Coimbra, Portugal \& Departamento de Matem\'atica, Universidade Federal da Pa\-ra\'\i \-ba, 58.051-900 Jo\~ao Pessoa, PB-Brazil}
\email{jmurb@mat.uc.pt}

\thanks{Part of this work was written during a visit of the authors  to the Hausdorff Research Institute for Mathematics (HIM), University of Bonn, in January  2019. The support and the hospitality of HIM are gratefully acknowledged.}
\thanks{EAP partially funded by CNPq-Brazil (Grants \#433623/2018-7, \#307500/2017-9), FAPERJ-Brazil (Grant \#E26/200.002/2018), Coordena\c{c}\~ao de Aperfei\c{c}oamento de Pessoal de N\'ivel Superior - Brasil (CAPES) - Finance Code 001 and PUC-Rio baseline funds.}
\thanks{JMU partially supported by FCT -- Funda\c c\~ao para a Ci\^encia e a Tecnologia, I.P., through projects PTDC/MAT-PUR/28686/2017 and UTAP-EXPL/MAT/0017/2017, and by the Centre for Mathematics of the University of Coimbra - UIDB/00324/2020, funded by the Portuguese Government through FCT/MCTES.}

\begin{abstract}
We study a nonlinear system made up of an elliptic equation of blended singular/degenerate type and Poisson's equation with a lowly integrable source. We prove the existence of a weak solution in any space dimension and, chiefly, derive an improved $\mathcal{C}^{1,\loglip}$ regularity estimate using tangential analysis methods. The system illustrates a sophisticated version of the proverbial thermistor problem and our results are new even in simpler modelling scenarios.
\bigskip


\bigskip

\noindent \textbf{AMS Subject Classifications MSC 2010:} 35B65, 35J57, 35J92, 35Q79.


\end{abstract}

\maketitle

\section{Introduction}

There are many good reasons to investigate the regularity properties of solutions to nonlinear partial differential equations (pdes) and systems. Perhaps the most compelling is the enhancement of more efficient numerical schemes leading to concrete applications of what would otherwise be a purely theoretical endeavour.  In recent years, there has been an intense activity around the development of a class of methods and techniques that culminate in the retrieval of improved regularity properties for solutions of a given pde imported from another pde which is \textit{somehow close} to the original one. This has been done, for example, in \cite{TU1}, where the sharp regularity for solutions of the inhomogeneous $p$-parabolic equation was derived from the regularity of $p$-caloric functions or in \cite{PT}, where Sobolev regularity for viscosity solutions of fully nonlinear elliptic equations was obtained from the associated recession profile. Other instances of this approach to regularity can be found, for example, in \cite{ATU1, ATU2, AMU, T1, T2, T3, T4, TU2}. For pointwise gradient bounds in terms of potentials, see also \cite{KM1, KM2, KM3}. 

In this paper, we extend this set of ideas, hitherto restricted to the analysis of single equations, to treat the nonlinear system of pdes 
\begin{equation}\label{bettercallsaul}
\left\{
\begin{array}{ll}
-\div \left(  \left| Du \right|^{\sigma \left( \theta (x) \right) -2} Du\right)   =  f \\
&  \\
-\Delta \theta  =   \lambda(\theta(x))  \left| Du \right|^{\sigma \left( \theta (x) \right)} ,
\end{array}
\right.
\end{equation}
proposed by Zhikov in \cite{Z07}, describing the steady state distribution of the electrical potential $u$ and the temperature $\theta$ in a \textit{thermistor}, a  portmanteau for a resistor whose electrical properties are thermally dependent. In \eqref{bettercallsaul}, $f$ is a given source and $\sigma$ and $\lambda$ are functions related to the electrical conductivity and resistance of the model, respectively. The thermistor problem, in its different versions, has been considered by many authors and there is an abundant literature around it, both in the physics/engineering and the mathematical communities. Far from being exhaustive, we mention here \cite{AC, GLS, HRS, R, Z07}.

The mathematical analysis of the strongly coupled system \eqref{bettercallsaul} involves two major difficulties. On the one hand, the second equation has a right-hand side merely in $L^1$ and this low integrability is known to be a source of severe analytical hazards. On the other hand, the first equation is not uniformly elliptic as its modulus of ellipticity collapses at points where $|Du|=0$, vanishing if the variable exponent $\sigma (\theta(x))$ is above two and blowing up if it is below that threshold. Since the exponent can vary in a range that crosses two, degeneracies and singularities are blended in our problem and one of the main achievements of the approach we use is to seamlessly treat the switching between regimes that otherwise correspond to two markedly different cases.

We first treat the existence of weak solutions under homogeneous Dirichlet boundary conditions. Contrary to the results of Zhikov in \cite{Z07}, which are valid only in dimensions up to three, we unlock the existence in any dimension. The key is in the use of a regularity result for the first equation in \eqref{bettercallsaul} that allows us to bypass the embedding related constraints surfacing in higher dimensions. We show in fact that a weak solution exists in the regularity class $(u,\theta) \in \mathcal{C}^{1,\beta} \times \mathcal{C}^{1,\alpha}$. We then improve the regularity for the temperature $\theta$, producing a $\mathcal{C}^{1,\loglip}$ local regularity result, with appropriate estimates.

The paper is organised as follows. In addition to specifying what we mean by a weak solution, we state in section \ref{salamanca} the assumptions on the data of the problem and present our main result. In section \ref{comboio}, we gather a few auxiliary results that will be instrumental in the sequel. The existence of weak solutions is established in section \ref{spicy}, by means of Schauder's Fixed Point Theorem. The final section \ref{call bell} brings the proof of the improved regularity.

\section{Assumptions and main result}\label{salamanca}

The system \eqref{bettercallsaul} holds in a given smooth domain $U\subset\mathbb{R}^d$, $d \geq 2$. We start with the formal definition of weak solution. To slightly assuage the notation, set
$$p:=\sigma \circ \theta \qquad {\rm and} \qquad a:=\lambda \circ \theta.$$ 

\begin{definition}\label{defweaksol}

A weak solution of \eqref{bettercallsaul}, coupled with homogeneous Dirichlet boundary conditions 
$$u=0 \qquad {\rm and} \qquad \theta=0 \qquad {\rm on} \ \partial U,$$
 is a pair 
$$(u, \theta) \in W_0^{1, p(\cdot)} (U) \times H_0^1(U)$$
such that
\begin{equation}\label{stanley}
\int_U \left| Du \right|^{p(x) -2} Du \cdot D\varphi = \int_U f \varphi, \quad \forall \, \varphi \in W_{0}^{1, p(\cdot)} (U)
\end{equation}
and
\begin{equation}\label{ermelinda-syrah}
\int_U D\theta \cdot D\psi = \int_U a(x) \left| Du \right|^{p(x)} \psi, \quad \forall \, \psi \in H_{0}^1(U) \cap L^\infty (U).
\end{equation}

\end{definition}

It is timely to comment again on the mathematical difficulties arising from this definition. First, we notice that $u\in W^{1,p(\cdot)}_0(U)$ leads to
\[
	\left|Du\right|^{p(\cdot)}\,\in\,L^1(U).
\]
As a consequence, the structure of \eqref{bettercallsaul} falls short in producing further regularity for the temperature $\theta$ since the integrability of the right-hand side of Poisson's equation does not even ensure the continuity of solutions. To circumvent this structural difficulty of the system, we will resort to an improved regularity result for the first equation in \eqref{bettercallsaul}. 

A second challenge comes from the eventual collapse of the ellipticity in \eqref{bettercallsaul}. Indeed, along $\left\lbrace Du\,=\,0\right\rbrace$, the first equation either degenerates or blows up, depending on $p(\cdot)$. In particular, we allow this variable exponent to oscillate around $p(\cdot)\equiv 2$; hence, the system may switch between the singular and the degenerate regimes. Our approach deals with this difficulty in a seamless fashion, which is in itself an unusual feature that deserves to be highlighted.
  
\medskip

We now list the main assumptions on the data of the problem. Throughout the paper, we say a constant is {\it universal} if it only depends on the data.

\begin{assumption}[Uniform bound on $\sigma$]\label{chickpeas}
The function $\sigma:\mathbb{R}\to\mathbb{R}$ is bounded from below and there exists a constant $\sigma^-$ such that
\[
	1\,\leq\,\frac{2d}{d+2}\,<\,\sigma^-\,\leq\,\sigma(x).
\]
\end{assumption}

The lower bound on $\sigma^-$ implies that solutions $u$ to the first equation in \eqref{bettercallsaul} are locally bounded. But, by considering a possibly unbounded (from above) function $\sigma$, we completely detach the analysis from the constant-exponent setting. Indeed, were $\sigma$ bounded and satisfying the next assumption, the variable exponent setting would not present extra challenges \textit{vis-a-vis} the constant case. 

\begin{assumption}[Lipschitz continuity of $\sigma$]\label{currylentils}
The function $\sigma$ is Lipschitz continuous and there exists $C_\sigma>0$ such that
\[
	\left\|\sigma\right\|_{\mathcal{C}^{0,1}(\mathbb{R})}\,\leq\,C_\sigma.
\]
\end{assumption}

Notice that A\ref{currylentils} does not imply $p(\cdot)$ to be a Lipschitz continuous exponent. Even if $\theta$ is H\"older continuous, the composite function $p=\sigma\circ\theta$ is, \textit{a priori}, merely H\"older continuous. 

\begin{assumption}[Uniform bounds on $\lambda$]\label{cologne}
The function $\lambda:\mathbb{R}\to\mathbb{R}$ is bounded and there exists a constant $\lambda^+$, to be fixed later and depending only on the data, such that
\[
	\left\|\lambda\right\|_{L^\infty(\mathbb{R})}\,\leq\,\lambda^+.
\]
\end{assumption}

The upper bound on $\lambda$ plays a critical role in the approximation methods put forward further in the paper. Perhaps more subtle is the fact that $\lambda^+$ is determined endogenously, in the context of the fixed-point argument in section \ref{spicy}.

\begin{assumption}[Integrability of the source]\label{shiraz}
The source term $f:U\to\mathbb{R}$ is in $L^\infty(U)$ and there exists $C_f>0$ such that
\[
	\left\|f\right\|_{L^\infty(U)}\,\leq\,C_f.
\]
\end{assumption}

We recall that a function is Log-Lipschitz continuous if it has a modulus of continuity of the type $\omega (\sigma)= \sigma \ln (1/\sigma)$. Since, for  $0<\gamma <1$, we have
$$\sigma \leq \sigma \ln (1/\sigma) \leq \frac{1}{(1-\gamma)e} \, \sigma^\gamma , \quad \forall \, 0<\sigma \leq \frac{1}{e},$$ 
this is weaker than Lipschitz continuity but implies the $C^{0,\gamma}$ H\"older continuity, for every $0<\gamma <1$. Observe that the constant on the right-hand side of the second inequality (which, in fact, holds for every $\sigma >0$) blows up as $\gamma \to 1^-$.

We now state the main theorem of our paper, the $\mathcal{C}^{1,\loglip}_{loc}(U)$-regularity for the temperature $\theta$.

\begin{theorem}\label{slicedbread}
Suppose A\ref{chickpeas}-A\ref{shiraz} are in force  and let $(u,\theta)$ be a weak solution to \eqref{bettercallsaul} such that
$$(u,\theta) \in W^{1,p(\cdot)}_0(U)\times\mathcal{C}^{0,\alpha}(\overline{U}).$$ 
Then, $\theta\in\mathcal{C}^{1,\loglip}_{loc}(U)$ and, for every compact $K\subset U$,  there exists $C>0$, depending only on the data and $K$, such that
\[
	\left\|\theta\right\|_{\mathcal{C}^{1,\loglip}(K)}\,\leq\,C.
\]
\end{theorem}

The theorem can be interpreted as follows: a slight refinement in the existence class resonates through the highly nonlinear coupling of the system to yield a substantial gain in regularity. 

\begin{remark}[Fully nonlinear variant]The tangential analysis techniques we will be using find application in a variety of distinct contexts. In particular, we believe that our arguments are flexible enough to accommodate a fully nonlinear variation of our problem, for example,
\begin{equation*}
	\begin{cases}
		\left|Du\right|^{\sigma(\theta(x))}F(D^2u,x)\,=\,f\\
		  \\
		G(D^2\theta,x)\,=\,\lambda(\theta(x))\left|Du\right|^{\sigma(\theta(x))+2},
	\end{cases}
\end{equation*}
where $F,\,G:\mathcal{S}(d)\times U\to\mathbb{R}$ are fully nonlinear elliptic operators. By designing an appropriate limiting profile, we would expect to establish improved regularity of the solutions in H\"older spaces.
\end{remark}

\section{A few auxiliary results}\label{comboio}

In this section, we gather previous developments and a few auxiliary results to which we resort in this paper.  We start with a lemma on the existence of weak solutions to the $p(x)$-Laplace equation.

\begin{lemma}[Existence of solutions to $p(x)$-Laplace equations]\label{indopak}
	Let $p:U\to\mathbb{R}$ be a continuous function and suppose $f\in L^\infty(U)$. Then, there exists a weak solution $u \in W_0^{1, p(\cdot)} (U) $ to
	\begin{equation}\label{coimbrab}
		\begin{cases}
			-\div\left(\left|Du\right|^{p(x)-2}Du\right)\,=\,f&\;\;\;\;\;\mbox{in}\;\;\;\;\;U\\
			u\,=\,0&\;\;\;\;\;\mbox{on}\;\;\;\;\;\partial U.
		\end{cases}
	\end{equation}
\end{lemma}

For the proof of Lemma \ref{indopak} we refer the reader to \cite{FZ03} (see also \cite{DHHR11}). In fact, finer regularity results for \eqref{coimbrab} are available and the next lemma concerns the H\"older-continuity of the gradient of the solutions. For a proof, we refer the reader to \cite{F07}. 

\begin{lemma}[Regularity of the solutions to $p(x)$-Laplace equations]\label{pedronunes}
Let $u\in W_0^{1,p(\cdot)}(U)$ be a weak solution to \eqref{coimbrab}. Suppose $p\in\mathcal{C}^{0,\alpha}(U)$ and $f\in L^\infty(U)$. Then,  $u\in\mathcal{C}^{1,\beta}(U)$, for some $\beta\in(0,1)$, and there exists $C>0$ such that
\[
	\left\|u\right\|_{\mathcal{C}^{1,\beta}(U)}\,\leq\,C,
\]
where $C=C(p,f)$ and $\beta=\beta(p,f)$.
\end{lemma}

We conclude this section with a result on the sequential stability of the weak solutions to \eqref{bettercallsaul}. First, it unlocks a continuity feature, required in the proof of the existence of solutions to \eqref{bettercallsaul}. Furthermore, sequential stability is pivotal to the approximation analysis underlying the improved regularity for our problem.

\begin{proposition}[Sequential stability of weak solutions]\label{vegan}
Suppose A\ref{chickpeas}-A\ref{shiraz} are in force. Let $(u_n,\theta_n)$ be a weak solution to 
\begin{equation}\label{bettercallsauln}
\left\{
\begin{array}{ll}
-\div \left(\left| Du_n \right|^{\sigma_n \left( \theta_n (x) \right) -2} Du_n\right)   =  f_n & {\rm in } \ U\\
&  \\
-\Delta \theta_n  =   \lambda_n \left( \theta_n (x) \right) \left| Du_n \right|^{\sigma_n \left( \theta_n (x) \right)} & {\rm in } \ U,
\end{array}
\right.
\end{equation}
where $(f_n)_{n\in\mathbb{N}}$, $(\sigma_n)_{n\in\mathbb{N}}$ and $(\lambda_n)_{n\in\mathbb{N}}$ are such that 
\[
	\left\|f_n\,-\,f\right\|_{L^\infty(U)}\,+\,\left\|\sigma_n\,-\,\sigma\right\|_{L^\infty(\R)}\,+\,\left\|\lambda_n\right\|_{L^\infty(\R)}\,\longrightarrow 0.
\]
Suppose further there exists $(u_\infty,\theta_\infty)$ such that $u_n\to u_\infty$ in $\mathcal{C}^{1,\beta}$ and $\theta_n\to\theta_\infty$ in $\mathcal{C}^{1,\beta}$, for some $\beta\in(0,1)$. Then,
\begin{equation*}\label{bettercallsaul_limit}
\left\{
\begin{array}{ll}
-\div \left( \left| Du_\infty \right|^{\sigma \left( \theta_\infty (x) \right) -2} Du_\infty\right)   =  f & {\rm in } \ U\\
&  \\
-\Delta \theta_\infty  =  0 & {\rm in } \ U.
\end{array}
\right.
\end{equation*}
\end{proposition}
\begin{proof}
The statement follows if we prove that 
\begin{equation}\label{ermelinda}
	\int_{U}\left|Du_\infty\right|^{\sigma(\theta_\infty(x))-2}Du_\infty\cdot D\phi \, dx\,=\,\int_{U}f\phi \, dx
\end{equation}
and 
\begin{equation}\label{quarteroni}
	\int_{U}D\theta_\infty\cdot D\psi \, dx\,=\,0,
\end{equation}
for every $\phi\in W^{1,p(\cdot)}_0(U)$ and every $\psi\in H^1_0(U)$. 

We will apply Lebesgue's Dominated Convergence Theorem and we start with \eqref{ermelinda}. Notice that
\begin{align*}
	&\int_{U}\left|Du_\infty\right|^{\sigma(\theta_\infty(x))-2}Du_\infty\cdot D\phi \, dx  =  \int_{U}f_n\phi \, dx\\
		&\qquad + \int_{U}\left(\left|Du_\infty\right|^{\sigma(\theta_\infty(x))-2}Du_\infty-\left|Du_n\right|^{\sigma_n(\theta_n(x))-2}Du_n\right)\cdot D\phi \, dx\\
		&\qquad:=\,\int_UI_n \, dx\,+\,\int_UJ_n \, dx.
\end{align*}
We first consider $J_n$. Since the sequence $(u_n)_{n\in\mathbb{N}}$ converges in $\mathcal{C}^{1,\beta}(U)$ we have that $\left|Du_n\right|$ is uniformly bounded. In addition, it follows from A\ref{currylentils} that $(\sigma_n \circ \theta_n)_{n\in\mathbb{N}}$ is also uniformly bounded. Hence,
\begin{align*}
	\left|J_n(x)\right|\,\leq\,C\left|D\phi(x)\right|\,\in\,L^1(U).
\end{align*}
It remains to verify that
\[
	\left|Du_n(x)\right|^{\sigma_n(\theta_n(x))-2}Du_n(x)\,\to\,\left|Du_\infty(x)\right|^{\sigma(\theta_\infty(x))-2}Du_\infty(x),
\]
as $n\to\infty$, i.e., that pointwise convergence indeed takes place. To that end, we distinguish two cases. 

First, let $x\in U$ be such that $Du_\infty(x)\neq 0$. Then, there exists $N\in\mathbb{N}$ such that $Du_n(x)\neq 0$ for every $n\geq N$. Therefore,
\begin{align*}
	&\left|\left|Du_\infty(x)\right|^{\sigma(\theta_\infty(x))-2}Du_\infty(x)-\left|Du_n(x)\right|^{\sigma_n(\theta_n(x))-2}Du_n(x)\right|\\
		&\quad\leq\, \left|Du_n(x)\right|^{\sigma_n(\theta_n(x))-2}\left|Du_n(x)-Du_\infty(x)\right|\\
		&\qquad\,+\left|\left|Du_n(x)\right|^{\sigma_n(\theta_n(x))-2}-\left|Du_\infty(x)\right|^{\sigma(\theta_\infty(x))-2}\right|\left|Du_\infty(x)\right|\\
		&\quad\leq\,o(1)\,+\,C\left|e^{\left(\sigma_n(\theta_n(x))-2\right)\ln\left|Du_n(x)\right|}-e^{\left(\sigma(\theta_\infty(x))-2\right)\ln\left|Du_\infty(x)\right|}\right|\,\to\,0,
\end{align*}
as $n\to\infty$. 

Consider next the case of $x\in U$ so that $Du_\infty(x)=0$. Now, 
\begin{align*}
	0\,&\leq\,\left|\left|Du_\infty(x)\right|^{\sigma(\theta_\infty(x))-2}Du_\infty(x)-\left|Du_n(x)\right|^{\sigma_n(\theta_n(x))-2}Du_n(x)\right|\\
		&\leq\, \left|Du_n(x)\right|^{\sigma_n(\theta_n(x))-1}\,+\,\left|Du_\infty(x)\right|^{\sigma(\theta_\infty(x))-1}\\
		&= \left|Du_n(x)\right|^{\sigma_n(\theta_n(x))-1}.
\end{align*}
Because $\left|Du_n\right|$ converges uniformly to $\left|Du_\infty\right|$, there exists $N\in\mathbb{N}$ such that $\left|Du_n(x)\right|<1/2$ for every $n>N$. We conclude that 
\begin{align*}
	0\leq&\left|\left|Du_\infty(x)\right|^{\sigma(\theta_\infty(x))-2}Du_\infty(x)-\left|Du_n(x)\right|^{\sigma_n(\theta_n(x))-2}Du_n(x)\right|\\
		&\qquad\,\leq\,\left|Du_n(x)\right|^{\sigma^--1}\,\to\,0,
\end{align*}
as $n\to\infty$. 

To treat $I_n$, notice that  
\[
	\left|f_n(x)\phi(x)\right|\,\leq\,C\left|\phi(x)\right|\,\in\,L^1(U).
\]
In addition, by assumption, we conclude that $f_n(x)\phi(x)\to f(x)\phi(x)$ for every $x\in U$, as $n\to\infty$. Therefore, Lebesgue's Dominated Convergence Theorem ensures that \eqref{ermelinda} holds true.

In what regards \eqref{quarteroni}, notice that $D\theta_n(x)\to D\theta_\infty(x)$, for every $x\in U$. Furthermore, the assumptions of the proposition yield
\[
	\left(\lambda_n \left( \theta_n (x) \right) \left| Du_n \right|^{\sigma_n \left( \theta_n (x) \right)}\right)\psi(x)\,\to\,0,
\]
as $n\to\infty$. A further application of Lebesgue's Dominated Convergence Theorem produces \eqref{quarteroni}. The proof of the proposition is then complete.

\end{proof}

The stability ensured by Proposition \ref{vegan} will play an instrumental role in our improved-regularity argument of section \ref{call bell}. 

\section{Existence of weak solutions}\label{spicy}

In this section, we show that the system \eqref{bettercallsaul} admits at least one weak solution in the sense of Definition \ref{defweaksol}. As is customary for this type of coupled system, we will use Schauder's Fixed Point Theorem. This strategy has been pursued before in the literature, \textit{e.g.}, in \cite{Z07}.

We emphasize two aspects of our argument. First, it bypasses previous dimensional constraints; this advance is due to an improved regularity result produced in \cite{F07}. In fact, suppose $\theta\in\mathcal{C}^{0,\beta}(U)$ is given; under A\ref{chickpeas} and A\ref{currylentils} the variable exponent $p:=\sigma\circ\theta$ is H\"older continuous and bounded (even if $\sigma$ is unbounded above). Together with the conditions on the source term, it frames the first equation in \eqref{bettercallsaul} within the scope of \cite[Theorems 1.1-1.2]{F07} and unveils the improved regularity for $u$. This gain-of-regularity mechanism opens room for standard compact imbedding results; see, for instance, \cite[Theorem 2.84]{DD12}.

\begin{proposition}[Existence of weak solutions]\label{rocinha}

Suppose A\ref{chickpeas}-A\ref{shiraz} are in force. Then, there exists a weak solution of \eqref{bettercallsaul} in the sense of Definition \ref{defweaksol}.

\end{proposition} 

\begin{proof}
The  proof is performed in two steps: we first define an operator and then show it fits the requirements of Schauder's Fixed Point Theorem.

\bigskip

\noindent {\bf Step 1.} We are going to apply Schauder's Fixed Point Theorem in $\mathcal{C}^{1, \alpha} (\overline{U})$, for some $0<\alpha<1$. Take $\theta^\ast \in \mathcal{C}^{1, \alpha} (\overline{U})$ and, using the results in \cite{DHHR11},  solve 

\begin{equation}\label{ines}
\left\{
\begin{array}{ll}
- \div \left( \left| Du \right|^{\sigma \left( \theta^\ast (x) \right) -2} Du\right)  = f & {\rm in } \ U\\
&  \\
u=0 &  {\rm on } \ \partial U,
\end{array}
\right.
\end{equation}
obtaining $u \in W_0^{1, \sigma  \circ \theta^\ast (\cdot)} (U)$. By Lemma \ref{pedronunes}, we have that $u \in \mathcal{C}^{1, \beta} (\overline{U})$, with a uniform estimate, and thus $|Du|$ is globally bounded and so is the right-hand side of 
$$-\Delta \theta  =   \lambda \left( \theta^\ast (x) \right) \left| Du \right|^{\sigma \left( \theta^\ast (x) \right)}.$$
We then solve this Poisson equation, with the Dirichlet condition $\theta = 0$ on $\partial U$, obtaining a solution
$$\theta \in W_0^{2,q} (U), \quad \forall \, 1<q<\infty.$$
In particular, by taking 
\[
	q\,:=\,\frac{2d}{1\,-\,\alpha},
\]
we obtain
\begin{equation}\label{couscous}
\theta \in W^{2,q} (U) \hookrightarrow \mathcal{C}^{1, \gamma} (\overline{U}),
\end{equation}
for every $\gamma < \frac{1+\alpha}{2}$ so, in particular, for $\gamma = \alpha$.
We define the operator 
$$\begin{array}{rccl}
{\mathcal T} : & \mathcal{C}^{1, \alpha} (\overline{U}) & \longrightarrow & \mathcal{C}^{1, \alpha} (\overline{U})\\
\\
& \theta^\ast & \longmapsto & {\mathcal T} (\theta^\ast) := \theta.
\end{array}
$$
It is obvious that a fixed point $\theta^\ast $ of ${\mathcal T}$ provides the local weak solution $(u, \theta^\ast)$ of \eqref{bettercallsaul}. 

\bigskip

\noindent {\bf Step 2.} The continuity of ${\mathcal T}$ is assured by continuous dependence results for the $p(x)$-Laplace equation and for Poisson's equation. In fact, if we take a sequence $\theta_n \to \theta_\infty$ in $\mathcal{C}^{1, \alpha} (\overline{U})$, then the corresponding variable exponents $\sigma (\theta_n (x) )$ converge uniformly. Hence, we can apply the results in \cite{ABO} (see also \cite{L}) to show that $u_n \to u_\infty$, first in the variable exponent Sobolev space and then, using the regularity result in Lemma \ref{pedronunes}, also in  $\mathcal{C}^{1, \beta} (\overline{U})$, for a certain $\beta$. Finally, standard continuous dependence results for Poisson's equation give the continuity of ${\mathcal T}$.

Since the compactness follows from \eqref{couscous}, it remains to show that ${\mathcal T}$ takes the unit ball in $\mathcal{C}^{1, \alpha} (\overline{U})$ into itself.  Let $\left\| \theta^\ast \right\|_{\mathcal{C}^{1, \alpha} (\overline{U})} \leq 1$. Then
$$\left| \sigma  \circ \theta^\ast (x) - \sigma  \circ \theta^\ast  (y) \right| \leq C_\sigma \left|  \theta^\ast (x) - \theta^\ast  (y) \right| \leq C_\sigma \left|  x - y \right| $$
and thus 
$$\left\| \sigma  \circ \theta^\ast\right\|_{\mathcal{C}^{0,1} (\overline{U})} \leq C_\sigma.$$
Then, by Lemma \ref{pedronunes}, we obtain
$$\left\| u \right\|_{\mathcal{C}^{1, \beta} (\overline{U})} \leq C_{data},$$
which gives a uniform control on the $L^\infty$-norm of the gradient $Du$. We then have, by classical elliptic regularity theory, 
\begin{eqnarray*}
\left\| \theta \right\|_{\mathcal{C}^{1, \alpha} (\overline{U})} & \leq & C_{emb} \left\|  \theta \right\|_{W^{2, \frac{2d}{1-\alpha}} (U)}\\
&  \leq & C_{emb} \left\|  \lambda \left( \theta^\ast (x) \right) \left| Du \right|^{\sigma \left( \theta^\ast (x) \right)} \right\|_{L^{\infty} (U)}\\
& \leq & C_{emb} \left\|  \lambda  \right\|_{L^{\infty} (U)} C_{data}\\
& \leq & 1,
\end{eqnarray*}
since $\left\|  \lambda  \right\|_{L^{\infty} (U)} \leq \frac{1}{C_{emb} C_{data}} =: \lambda^+$ due to A\ref{cologne}.

\end{proof}

\section{Improved regularity}\label{call bell}

Proposition \ref{rocinha} establishes the existence of a weak solution, which further belongs to the regularity class
$$(u,\theta)\in\mathcal{C}^{1,\beta}(\overline{U})\times\mathcal{C}^{1,\alpha}(\overline{U}).$$ 
We notice that, were we given a solution in $W^{1,p(\cdot)}_0(U)\times\mathcal{C}^{0,\alpha}(\overline{U})$, elliptic regularity theory and Lemma \ref{pedronunes} would lead to the same levels of regularity. In this section, we make this assumption to improve the regularity for $\theta$.

The tangential analysis methods we will use operate in two distinct layers: they first finely combine stability and compactness for solutions and then localize the analysis through a scaling argument, unveiling the geometric properties of the problem. Accordingly, we proceed by producing an approximation result for the solutions to \eqref{bettercallsaul}.

\begin{proposition}[Approximation Lemma]\label{prop_stanley} Suppose A\ref{chickpeas}-A\ref{shiraz} are in force and let 
$$(u,\theta) \in W^{1,p(\cdot)}_0(U)\times\mathcal{C}^{0,\alpha}(\overline{U})$$ 
be a solution to \eqref{bettercallsaul}. Given $\varepsilon>0$, there exists $\delta>0$ such that, if
\[
	\left\|\lambda\right\|_{L^\infty(U)}\,<\,\delta,
\]
then there exists a harmonic function $h\in\mathcal{C}^\infty (U)$ satisfying
\begin{equation*}\label{pocoyo}
	\left\|\theta-h\right\|_{L^\infty(U)}\,<\,\varepsilon.
\end{equation*}
\end{proposition}
\begin{proof}
We argue by a contradiction argument; suppose the statement of the Proposition is false. Then there exist $\varepsilon_0>0$ and sequences of functions $(u_n)_{n\in\mathbb{N}}$, $(\theta_n)_{n\in\mathbb{N}}$ and $(\lambda_n)_{n\in\mathbb{N}}$ such that
\[
	\,\left\|\lambda_n\right\|_{L^\infty(\mathbb{R})}\,\longrightarrow 0
\]
and
\begin{equation*}
	\begin{cases}
		-\div\left(\left|Du_n\right|^{\sigma(\theta_n(x))-2}Du_n\right)\,=\,f&\;\;\;\;\;\mbox{in}\;\;\;\;\;U\\
		& \\
		-\Delta \theta_n\,=\,\lambda_n\left(\theta_n\right)\left|Du_n\right|^{\sigma(\theta_n)}&\;\;\;\;\;\mbox{in}\;\;\;\;\;U,
	\end{cases}
\end{equation*}
but
\[
	\left\|\theta_n\,-\,h\right\|_{L^\infty(U)}\,>\,\varepsilon_0,
\]
for every $n\in\mathbb{N}$ and $h\in\mathcal{C}^\infty(U)$ harmonic. 

From the regularity theory available for $(u_n,\theta_n)$, we infer the existence of a pair $(u_\infty,\theta_\infty)$ such that 
\[
	\left\|u_n\,-u_\infty\right\|_{\mathcal{C}^{1,\beta}(U)}\,+\,\left\|\theta_n\,-\,\theta_\infty\right\|_{\mathcal{C}^{1,\beta}(U)}\,\longrightarrow \,0,
\]
as $n\to\infty$. By the stability of weak solutions, Proposition \ref{vegan}, we conclude that $\theta_\infty$ satisfies
\[
	\Delta\theta_\infty\,=\,0\;\;\;\;\;\mbox{in}\;\;\;\;\;U
\]
and hence, $\theta_\infty\in\mathcal{C}^\infty(U)$ is harmonic. Therefore, by setting $h:=\theta_\infty$, we reach a contradiction and the proof is complete.

\end{proof}

In the next proposition, we produce an oscillation control for the difference of $\theta$ and a paraboloid $P(x)$.

\begin{proposition}\label{maia}
Suppose A\ref{chickpeas}-A\ref{shiraz} are in force and let 
$$(u,\theta) \in W^{1,p(\cdot)}_0(U)\times\mathcal{C}^{0,\alpha}(\overline{U})$$ 
be a solution to \eqref{bettercallsaul}. Then, there exists a universal constant $0\,<\,\rho\,\ll\,1$ such that
\[
	\sup_{B_\rho}\,\left|\theta(x)\,-\,\left(a\,+\,b\cdot x\,+\,\frac{1}{2}x^TMx\right)\right|\,\leq\,\rho^{2},
\]
for some $a\in\mathbb{R}$, $b\in\mathbb{R}^d$ and $M\in\mathcal{S}(d)$.
\end{proposition}
\begin{proof}
Take $0< \varepsilon <1$, to be determined further down, and use Proposition \ref{prop_stanley} to obtain a harmonic function $h\in\mathcal{C}^\infty (U)$ such that 
\[
	\left\|\theta - h\right\|_{L^\infty(U)} < \varepsilon.
\]
Note that we have a universal control on the $L^\infty$-norm of $h$ since
\[
	\left\|h\right\|_{L^\infty(U)} \leq \left\|\theta - h\right\|_{L^\infty(U)} + \left\|\theta \right\|_{L^\infty(U)} < \varepsilon +M \leq C.
\]

We point out that a standard scaling argument (see, for example \cite{ATU1, T3}) puts us in the (then unrestrictive) smallness regime required by Proposition \ref{prop_stanley}. Indeed, given a solution $(u,\theta)$ to \eqref{bettercallsaul}, define
\[
	\theta_K(x)\,:=\,\frac{\theta(x)}{K},
\]
for some $K>0$ to be fixed. Notice the pair $(u,\theta_K)$ solves 

\begin{equation}\label{everest}
	\begin{cases}
		-\div\left(\left|Du\right|^{\sigma_K(\theta_K(x))-2}Du\right)\,=\,f\\
		\\
		-\Delta\theta_K\,=\,\lambda_K \left(\theta_K(x)\right)\left|Du\right|^{\sigma_K(\theta_K(x))},
	\end{cases}
\end{equation}
where
\[
	\sigma_K(t)\,:=\,\sigma(Kt)\;\;\;\;\;\;\;\;\;\;\;\;\mbox{and}\;\;\;\;\;\;\;\;\;\;\;\;\lambda_K(t)\,:=\,\frac{\lambda(Kt)}{K}.
\]
It is straightforward to check that \eqref{everest} satisfies A\ref{chickpeas}-A\ref{shiraz}. Hence, by choosing
\[
	K\,:=\,\frac{\left\|\lambda\right\|_{L^\infty(U)}}{\delta},
\]
we fall into the required smallness regime.

We then have
\begin{align*}
	\sup_{B_\rho}&\left|\theta(x)-h(0)-Dh(0)\cdot x-\frac{x^TD^2h(0)x}{2}\right|\leq\,\sup_{B_\rho}\left|\theta(x)-h(x)\right|\\
		&\qquad+\sup_{B_\rho}\left|h(x)-h(0)-Dh(0)\cdot x-\frac{x^TD^2h(0)x}{2}\right|\\
		&\qquad\leq \varepsilon\,+\,C\rho^{2+\alpha},
\end{align*}
for some $\alpha\in(0,1)$, where $C>0$ is a universal constant (here the harmonicity of $h$ plays a crucial role). Define 
\[
	\rho\,:=\,\left(\frac{1}{2C}\right)^\frac{1}{\alpha}\;\;\;\;\;\;\;\;\;\;\mbox{and}\;\;\;\;\;\;\;\;\;\;\varepsilon\,:=\,\frac{\rho^{2}}{2}.
\]
By setting $a:=h(0)$, $b:=Dh(0)$ and $M:=D^2h(0)$, we conclude the proof.

\end{proof}

Observe that the constant matrix $M=D^2h(0)$ in the proof of Proposition \ref{maia} satisfies $\tr(M)=0$. The next result is a discrete counterpart of Proposition \ref{maia}, at the scale $\rho^n$, for $n\in\mathbb{N}$. 

\begin{proposition}\label{raimundao}
Suppose A\ref{chickpeas}-A\ref{shiraz} are in force and let 
$$(u,\theta) \in W^{1,p(\cdot)}_0(U)\times\mathcal{C}^{0,\alpha}(\overline{U})$$ 
be a solution to \eqref{bettercallsaul}.Then, there exists a sequence of polynomials $(P_n)_{n\in\mathbb{N}}$ of the form
\[
	P_n(x)\,:=\,a_n\,+\,b_n\cdot x\,+\,\frac{x^TM_nx}{2},
\]
satisfying
\begin{equation}\label{poppelsdorf}
	\tr\left(M_n\right)\,=\,0,
\end{equation}

\begin{equation}\label{ibis}
	\sup_{B_{\rho^n}}\,\left|\theta(x)\,-\,P_n(x)\right|\,\leq\,\rho^{2n}
\end{equation}
and
\begin{equation}\label{felix}
	\left|a_n-a_{n-1}\right|+\rho^{n-1}\left|b_n-b_{n-1}\right|+\rho^{2(n-1)}\left|M_n-M_{n-1}\right|\leq C\rho^{2(n-1)},
\end{equation}
for every $n\in\mathbb{N}$.

\end{proposition}

\begin{proof}
The result follows from an induction argument. The statement of Proposition \ref{maia} accounts for the case $n=1$. Suppose the case $n=k$ has already been verified. We consider the case $n=k+1$.

Consider the auxiliary function
\[
	v_k(x)\,:=\,\frac{\theta(\rho^k x)\,-\,P_k(\rho^k x)}{\rho^{2k}}.
\]
Notice that $v_k$ solves
\[
	-\Delta v_k\,=\,\lambda\left(\rho^{2k}v_k(x)+P_k(\rho^k x)\right)\left|Du(\rho^k x)\right|^{\sigma\left(\rho^{2k}v_k(x)+P_k(\rho^k x)\right)}.
\]
Hence, by imposing a suitable smallness regime on the $L^\infty$-norm of $\lambda(x)$, $v_k$ falls within the scope of Proposition \ref{prop_stanley}. Therefore, there exists $\overline{h}\in\mathcal{C}^\infty(U)$ such that
\begin{equation}\label{niewpoort}
	\sup_{B_\rho}\,\left|v_k(x)\,-\,\overline{h}(0)\,-\,D\overline{h}(0)\cdot x\,-\,\frac{x^TD^2\overline{h}(0)x}{2}\right|\,\leq\,\rho^{2}.
\end{equation}
Set $P_{k+1}(x)$ as
\[
	P_{k+1}(x)\,:=\,a_k\,+\,\rho^{2k}\overline{h}(0)\,+\,\left(b_k\,+\,\rho^kD\overline{h}(0)\right)\cdot x\,+\,\frac{x^T\left(M_k\,+\,D^2\overline{h}(0)\right)x}{2}.
\]
It follows from \eqref{niewpoort} that
\[
	\sup_{B_{\rho^{k+1}}}\,\left|\theta(x)\,-\,P_{k+1}(x)\right|\,\leq\,\rho^{2(k+1)}
\]
and, in addition, 
\[
	\tr\left(M_k\,+\,D^2\overline{h}(0)\right) =0.
\]
Finally, we also have
\[
	\left|a_{k+1}\,-\,a_{k}\right|\,=\,\rho^{2k}\left|\overline{h}(0)\right|,
\]
\[
	\left|b_{k+1}\,-\,b_{k}\right|\,=\,\rho^{k}\left|D\overline{h}(0)\right|
\]
and
\[
	\left|M_{k+1}\,-\,M_k\right|\,=\,\left|D^2\overline{h}(0)\right|.
\]
Hence, 
\[
	\left|a_{k+1}-a_k\right|+\rho^{k}\left|b_{k+1}-b_k\right|+\rho^{2k}\left|M_{k+1}-M_k\right|\leq C\rho^{2k},
\]
where $C>0$ is a universal constant.

\end{proof}

Observe that, from \eqref{ibis} and  \eqref{felix}, we have 
$$a_n \rightarrow \theta(0), \qquad b_n \rightarrow D\theta(0),$$
with 
\begin{equation}\label{flughaven}
\left| a_n - \theta(0) \right| \leq \rho^{2n} , \qquad \left| b_n - D\theta(0)  \right| \leq C\rho^{n}.
\end{equation}
Although we can not conclude about the convergence of the sequence of matrices $(M_n)_n$, we still obtain, again from \eqref{felix}, the estimate 
\begin{equation}\label{precocerto}
\left| M_n \right| \leq Cn.
\end{equation}

We conclude the paper with the proof of Theorem \ref{slicedbread}, which amounts to producing the continuous version of Proposition \ref{raimundao}.

\begin{proof}[Proof of Theorem \ref{slicedbread}]Let $0<r \leq \rho \ll1$ be given. Take $n\in\mathbb{N}$ such that $\rho^{n+1} < r \leq \rho^n$. Observe that then
$$n \leq \dfrac{\ln r}{\ln \rho}.$$

From Proposition \ref{raimundao} and estimates \eqref{flughaven} and \eqref{precocerto}, we obtain
\begin{align*}
	&\sup_{B_r}\left|\theta(x) - (\theta(0) + D\theta(0)\cdot x )\right|\,\leq\,  \sup_{B_{\rho^n}} \left| \theta(x) - (\theta(0) + D\theta(0)\cdot x )\right|\\
	&\qquad\,\leq\, \sup_{B_{\rho^n}}\left| (\theta - P_n) + a_n - \theta(0)+b_n\cdot x - D\theta(0)\cdot x + \dfrac{x^tM_nx}{2} \right|\\
 	&\qquad\,\leq\,\rho^{2n} + \rho^{2n} + C\rho^{2n} + \dfrac{C}{2}n\rho^{2n} \\
	&\qquad\,\leq\, C \left( \rho^{2n} + n\rho^{2n} \right) \\
	&\qquad\,\leq\,  \dfrac{C}{\rho^2} \left( \rho^{2(n+1)} + n\rho^{2(n+1)} \right) \\
	&\qquad\,\leq\,\dfrac{C}{\rho^2} \left( r^2 +\dfrac{\ln r}{\ln \rho} \, r^2 \right) \\
	&\qquad\,\leq\,\dfrac{C}{\rho^2 |\ln \rho|} \left( |\ln \rho| +\ln \frac{1}{r}\right) r^2 \\
	&\qquad\,\leq\,\dfrac{2C}{\rho^2 |\ln \rho|}\, r^2 \ln \frac{1}{r} \\
	&\qquad\,\leq\, Cr^2\ln \frac{1}{r},
\end{align*}
for a universal constant $C>0$. The conclusion that, locally, $\theta$ has continuous first order derivatives, with a Log-Lipschitz modulus of continuity, follows from standard arguments in regularity theory.

\end{proof}

\begin{remark}
If we further assume the function $\lambda$ to be H\"older continuous we enter the realm of Schauder's regularity theory and therefore
$$\theta \in \mathcal{C}^{2,\alpha}(U),$$
for some $\alpha \in (0,1)$.
\end{remark}

\bigskip

\end{document}